\documentclass [11pt,oneside]{amsart}

\reversemarginpar \textwidth=14cm \textheight=19cm

\usepackage{amssymb} \usepackage{amsfonts} \usepackage{amsmath}

\usepackage{amsthm} 
\usepackage{epsfig} 
\usepackage{tikz}
\usepackage[cmtip,arrow]{xy}
\usepackage{amsmath,amssymb,enumerate,pb-xy}

\usepackage[applemac]{inputenc}

\usepackage{mathrsfs} 
\usepackage{verbatim} 
\usepackage{hyperref} 

\usepackage[small,nohug,heads=littlevee]{diagrams} 
\diagramstyle[labelstyle=\scriptstyle]

\input xy
\xyoption{all}

\usepackage{caption}

\newtheorem{dummy}{Dummy}[section]

\newtheorem{prop}[dummy]{Proposition}

\theoremstyle{definition}

\newtheorem{theorem}[dummy]{Theorem}
\newtheorem{lem}[dummy]{Lemma}
\newtheorem{cor}[dummy]{Corollary} 

\theoremstyle{definition}

\newtheorem{dfn}[dummy]{Definition}

\newtheorem{rem}[dummy]{Remark}
\newtheorem{ntz}[dummy]{Notation}

\theoremstyle{remark}

\newcommand{\N}{\mathbb{N}}

\newcommand{\Z}{\mathbb{Z}}

\newcommand{\lip}{\mathrm{{Lip\,}}}
\newcommand{\lf}{\mathrm{{lf}}}
\newcommand{\Lip}{\mathrm{{Lip\,}}}

\newcommand{\diam}{\mathrm{{diam}}}

\newcommand{\isom}{\mathrm{Isom}}
\newcommand{\smear}{\mathrm{smear}}

\newcommand{\st}{\mathrm{pst}}
\newcommand{\pst}{\mathrm{pst}}
\renewcommand{\sp}{\mathrm{pst}}

\renewcommand{\:}{\mathrm{\,:\,}}

\newcommand{\id}{\mathrm{\,Id\,}}
\newcommand{\supp}{\mathrm{\,Supp\,}}

\newcommand{\todo}[1]{\vspace{5mm}\par\noindent
\framebox{\begin{minipage}[c]{0.95 \textwidth} \tt #1
\end{minipage}} \vspace{5mm} \par}

\newcommand{\R} {\ensuremath {\mathbb{R}}}

\newcommand{\sm} {\ensuremath {{\rm sm}}}

\newcommand{\vol} {{\rm Vol}\,}
\newcommand{\dvol} {{\rm dVol}\,}

\author{Federico Franceschini}

\address{Dipartimento di Matematica \\
Universit\`a di Pisa \\
Largo B.~Pontecorvo 5 \\
56127 Pisa, Italy}

\email{franceschini@dm.unipi.it}

\title{Proportionality Principle for the Lipschitz simplicial volume}


\thanks{}

\begin{document}

\maketitle

\begin{abstract} We prove the Proportionality Principle for the Lipschitz simplicial volume without any restriction on curvature. Our argument is based on the construction of a suitable \emph{pseudostraightening} for simplices. 
\end{abstract}



\section{\label{sezione 1} Introduction and outline of the paper}

The simplicial volume  
is an invariant of manifolds introduced by Gromov in his seminal paper~\cite{G}. If $M$ is a connected, compact, oriented manifold, 
then the simplicial volume $\| M\|$ of $M$ is the 
infimum of the sum of the absolute values of the coefficients over all singular chains representing the real fundamental cycle of $M$ (see Section~\ref{sezione 2}). 
If $M$ is $n$-dimensional and open, then the $n$-dimensional homology of $M$
vanishes. Then, in order to get a well-defined notion of fundamental class, whence
of simplicial volume, one needs to refer to \emph{locally finite} chains.

Even if it depends only on the proper homotopy type of a manifold, the simplicial volume is deeply related
to the geometric structures that a manifold can carry. For example, closed manifolds which support 
negatively curved Riemannian metrics have nonvanishing simplicial volume, while the simplicial 
volume of flat or spherical manifolds is null (see \emph{e.g.}~\cite{G}).
An important result which relates the simplicial volume of a manifold
to its geometry is \emph{Gromov Proportionality Principle}:

\begin{theorem} [\cite{G}] 
\label{classical Proportionality Principle} Let $M$ and $N$ be 
\emph{compact} Riemannian manifolds sharing the same Riemannian universal covering. Then 
$$ \frac{\|M\|}{\vol M} = \frac{\|N\|}{\vol N}. $$
\end{theorem}

There exist two different approaches to Theorem~\ref{classical Proportionality Principle} available in the literature. 
The first one, which is closer to Gromov's original discussion of the issue, relies 
on the study of the (continuous) bounded cohomology of $M$ (see e.g.~\cite{Bucher, Frigerio}). The other one,
which is due to Thurston~\cite{T}, is based on the introduction of the \emph{measure homology} of $M$, and leads to the first detailed proof of 
the proportionality principle, which was given by L\"oh in~\cite{Larticolo}.
However, it may be worth mentioning that, in~\cite{Larticolo}, 
the proof of the fact that measure homology is isometric to
the standard singular homology, which is the key step towards the proportionality principle, still relies on results about bounded cohomology.

As observed in~\cite{LS}, Gromov Proportionality Principle does not hold 
in the context of open manifolds: 
a deep (and difficult) result by Gromov shows that the simplicial volume of the 
product of three open manifolds necessarily vanishes~\cite[Example (a), p. 59]{G}, 
so the simplicial volume of the product of three cusped hyperbolic surfaces is zero; 
on the other hand, it is well-known that, in the case of closed manifolds, the simplicial volume is supermultiplicative (see again~\cite[Section 0.2]{G}),
so the product of the simplicial volume of three \emph{closed} hyperbolic surfaces is positive. 
But both the product of three cusped hyperbolic surfaces and the product of three closed hyperbolic surfaces
are covered by the product of three 
copies of the hyperbolic plane. This implies that the proportionality principle cannot hold for open manifolds.
In fact, the geometric meaning of the simplicial volume of open manifolds is still 
quite mysterious: for example, it is not known whether it vanishes even in the simple case of the
product of two one-holed tori (see also \cite{BKK} for more (non) vanishing results for the simplicial volume of non-compact manifolds).

In order to circumvent these difficulties, Gromov introduced the notion of \emph{Lipschitz simplicial volume}~\cite{G}. Roughly speaking,
the Lipschitz simplicial volume $\|M\|_\lip$ of an open manifold $M$ is obtained by minimizing the sum of the absolute values of the coefficients
over locally finite real fundamental cycles whose simplices have a uniformly bounded Lipschitz constant (see Section \ref{sezione 2} for the precise definition).
L\"oh and Sauer then proved the Gromov Proportionality Principle for the Lipschitz simplicial volume of \emph{non-positively curved}
Riemannian manifolds. 

The purpose of the present work is to generalize L\"oh's and Sauer's results in order to drop any curvature condition. In fact, our main result is the following

\begin{theorem} [Main theorem] \label{main theorem} Let $M$ and $N$ be complete finite-volume Riemannian manifolds with isometric universal covers. Then:
$$ \frac{\|M\|_{\Lip}}{\vol M} = \frac{\|N\|_{\Lip}}{\vol N}. $$
\end{theorem}

In proving our main theorem, we will provide a new and self--contained proof also of the classical Proportionality Principle. Surprisingly enough, our proof does not make use of
the powerful machinery of bounded cohomology (which, as mentioned above, plays an important role in the available proofs of the proportionality principle in the compact case).
A slightly weaker result was recently obtained by K.~Strza{\l}kowski in~\cite{polacco}, who proved Theorem~\ref{main theorem} under the additional assumption
that the curvature of $M$ is bounded above.

Using results given in Section \ref{sezione 3}, we will also be able to prove the following result:

\begin{theorem} \label{productEstimates} Let $M$ and $N$ be complete Riemannian manifolds. Then:
$$ \|M\|_\lip \|N\|_\lip \le \|M \times N\|_\lip \le \binom{\dim M + \dim N}{\dim M} \|M\|_\lip \|N\|_\lip. $$
\end{theorem}

Theorem \ref{productEstimates} generalizes an analogous proposition in \cite[Theorem 1.7]{LS}, where $M$ and $N$ are assumed to be non-positively curved.




Along the way, we also prove that the Lipschitz simplicial volume can be computed just by looking at \emph{smooth} chains (see Corollary \ref{3o7hfuweidj}).

\subsection*{Acknowledgements} This work is part of a Ph.~D.~project that I am developing under the supervision of Roberto Frigerio. I would like to thank him for many precious conversations about the subject. The author also thanks Clara L\"oh and Roman Sauer for useful comments on a preliminary version of the paper.

\section{\label{sezione 2} Basic definitions and strategy of the proof}

Unless otherwise stated, in this paper all the (co)chains modules and the (co)homology modules will be understood
with (trivial) real coefficients. Therefore, if $X$ is any topological space, we will denote by
$C_*(X)$ (resp.~$C^*(X)$) and $H_*(X)$ (resp.~$H^*(X)$) the singular (co)chains (resp.~singular (co)homology) of $X$ with real coefficients.

We say that a 
a set of $i$--simplices $\{\sigma_k\}_{k \in \N}\subseteq S_i(X)$ is \textbf{locally finite} if every compact subset of $X$ 
intersects the image of $\sigma_k$ only for a finite number of indices. For every formal sum
$c=\sum_{k\in\N} \lambda_k\sigma_k$, where $\sigma_k\in S_i(X)$ and $\sigma_k\neq \sigma_h$ if $h\neq k$, we define the \textbf{support}
of $c$ by setting $\supp(c)=\{\sigma_k\, |\, \lambda_k\neq 0\}\subseteq S_i(X)$.

Then, the \textbf{complex of locally finite chains $C_*^{\lf}(X)$} is defined as follows: 
$$ C_i^{\lf}(X) := \left\{c=\sum_{k \in \N} \lambda_k \sigma_k \: \supp(c) \textrm{ is a locally finite family } \right\}\ . $$ 
It is easy to check that the usual boundary operator on finite chains extends to locally finite ones, so $C_*^{\lf}(X)$ is indeed
a differential complex. 
We denote by $H_*^{\lf}(X)$ the corresponding homology. Of course, if $X$ is compact, then $C_*(X)=C_*^{\lf}(X)$ and
$H_*(X)=H_*^{\lf}(X)$.

When $M$ is a Riemannian manifold, we may restrict our attention to locally finite chains satisfying an additional regularity property.
For $\sigma \in S_i(M)$ we denote by $\lip (\sigma)\in [0,\infty]$ the Lipschitz constant of $\sigma$, where
it is understood that $L(\sigma)=\infty$ if $\sigma$ is not Lipschitz, and for 
a locally finite chain $c\in C_*^{\lf}(M)$ we set
$$ \lip (c) := \sup \{\lip(\sigma) \: \sigma\in \supp(c)\}\in [0,\infty] \ .$$
If $\lip (c) < \infty$ we say that $c$ is a \textbf{Lipschitz chain}. 
Locally finite Lipschitz chains provide a 
subcomplex $C_*^{\lf, \lip}(M)$ of $C_*^{\lf}(M)$, whose associated homology will be denoted by $H_*^{\lf, \lip}(M)$.

In order to define the (Lipschitz) simplicial volume, we also need to put on (locally finite) chains an $\ell^1$-norm $\|\cdot\|_1$, which is defined as follows:
for every $c=\sum_{k\in\N} \lambda_k\sigma_k\in C_i^{\lf}(M)$ such that $\sigma_k\neq \sigma_h$ for $h\neq k$ we set
$$
\|c\|_1=\sum_{k\in\N} |\lambda_k|\ \in\ [0,\infty].
$$
This (possibly infinite) norm restricts to $C_i(M)$ and $C_i^{\lf,\lip}(M)$. By taking the infimum over representatives, these norms define 
(possibly infinite) semi-norms (still denoted by $\|\cdot\|_1$)
on the homology modules $H_i(M)$, $H_i^{\lf}(M)$, $H_i^{\lf,\lip}(M)$. Moreover, we define $C_*^{\lf, \ell^1, \lip}(M)$ as the subcomplex of $C_*^{\lf,\lip}(M)$ whose chains have finite $\ell^1$ norm.

Let now $M$ be an $n$-dimensional Riemannian manifold. In this paper, every manifold will be assumed to be connected and oriented. 
Of course, one may define locally finite and Lipschitz chains also in the context of singular chains with \emph{integral} coefficients. Then,
it is well-known that the $n$-dimensional locally finite homology module of $M$ is isomorphic to $\Z$, and generated by the so--called \emph{integral fundamental class} of $M$.
We denote by $[M]\in H_n^{\lf}(M)$ the \emph{real fundamental class} of $M$, i.e.~the image of the integral fundamental class under the change of coefficients homomorphism. 
We are now ready to define the simplicial volume of $M$.

\begin{dfn}[\cite{G}]
 If $M$ is an $n$-dimensional manifold, then the \textbf{simplicial volume} $\|M\|$ of $M$ is given by
 $$
 \| M\|=\| [M]\|_1\ .
 $$
\end{dfn}

Of course, if $M$ is compact, then $\|M\|<\infty$. However, the simplicial volume of open manifolds may be infinite.
Before defining the Lipschitz simplicial volume, we recall the following result: 



\begin{theorem}[\cite{LS}, Theorem 3.3]\label{LS1}
 Let $M$ be a connected Riemannian manifold. Then the homomorphism
 $$
 H_*^{\lf,\lip}(M)\to H_*^{\lf}(M)
 $$
 induced by the inclusion $C_*^{\lf,\lip}(M)\to C_*^{\lf}(M)$ is an isomorphism. 
\end{theorem}

Therefore, we can denote by $[M]_{\lip}$ the \textbf{Lipschitz fundamental class} of $M$, i.e. the element
of $H_*^{\lf,\lip}(M)$ corresponding to $[M]$ under the isomorphism provided by Theorem~\ref{LS1}. 
The following definition was originally given in \cite[Section 4.4f]{G}, and introduces the main object of study of this paper.

\begin{dfn}[\cite{LS}] \label{lipvolume} Let $M$ be a Riemannian manifold. Then the \textbf{Lipschitz simplicial volume} of $M$ is given by
 $$
 \|M\|_\lip=\|[M]_\lip\|_1\ .
 $$
\end{dfn}

Of course we have $\|M\|_\lip \geq \|M\|$ for every Riemannian manifold $M$. However, it may well be that the equality
does not hold: for example if $M$ is the product of three hyperbolic surfaces, then $\|M\|_{\lip}>0$ by L\"oh and Sauer Proportionality Principle,
while $\|M\|=0$ by Gromov's vanishing theorem for the product of three open manifolds.


\subsection*{Strategy of the proof}
Let us now come to the situation we are interested in, i.e.~let $M$ and $N$ be complete finite-volume Riemannian manifolds sharing the same
universal covering $U$. In order to compare the (Lipschitz) simplicial volume of
$M$ with the (Lipschitz) simplicial volume of $N$ it is necessary to 
produce a (Lipschitz) fundamental cycle for $N$ out of 
a (Lipschitz) fundamental cycle of $M$. In doing this, one also needs to keep control of the $\ell^1$-norm of the resulting cycle.
Of course, one may lift a fundamental cycle for $M$ to a fundamental cycle for $U$ which is invariant with respect to the action of
$\pi_1(M)$ on $U$. In order to project this cycle onto a cycle on $N$, however, the invariance with respect to the action of
$\pi_1(N)$ is needed. To this aim, Thurston introduced a \emph{smearing} procedure, which allows to average any cycle on $U$
with respect to the action of the full group of orientation-preserving isometries on $U$. The resulting object is obviously $\pi_1(N)$-invariant,
but it is no more a genuine (locally finite) cycle: in fact, the smearing of a cycle is a so-called \emph{measure cycle}. 
Therefore, in order to conclude the proof of the Proportionality Principle it is necessary to show that 
the simplicial volume may be computed in terms of a suitably defined seminorm on measure homology. Equivalently, one should show
that the existence of a measure fundamental cycle of a given norm implies the existence of a genuine (Lipschitz locally finite)
fundamental cycle whose norm approximates arbitrarily well the $\ell^1$-norm of the measure cycle.

In the compact case, the details of the proof just sketched were filled in by L\"oh, who proved that singular homology and measure homology are isometrically isomorphic~\cite{Larticolo}.
In the non-compact non-positively curved case, L\"oh and Sauer defined in~\cite{LS} a \emph{straightening} operator on simplices, which can be exploited to
turn measure cycles into genuine (Lipschitz locally finite) cycles, thus proving Theorem~\ref{main theorem} under the hypothesis that the manifolds involved are non-positively
curved. In this paper we show how to define a \emph{pseudostraightening} operator for Lipschitz simplices in manifolds without any curvature bound. In fact, the key result of this paper
is given by Theorem~\ref{pseudostraight:thm} below, which is proved in Section~\ref{sezione 3}, and is probably of independent interest. For technical reasons, that will be apparent later, the theorem pays a particular attention to smooth simplices. Before stating it, let us introduce some notation that will be extensively used later on.


If $Y$ and $X$ are metric spaces, we denote by $\lip(Y, X)$ the set of Lipschitz maps from $Y$ to $X$. We endow $\Delta^i$ and $\Delta^{i} \times I$ with the Euclidean metrics that they inherit as subspaces of $\R^i$ and $\R^{i+1}$ respectively. If $X$ is a Riemannian manifold and $L\geq 0$, we set $S_i^L(X)=\{\sigma\in S_i(X)\, |\, \lip(\sigma)\leq L\}$, and $S_i^\lip(X) = \{\sigma\in S_i(X)\, |\, \lip(\sigma)<\infty\} = \lip(\Delta^i, X)$. On $S_i(X)$ we put the metric of uniform distance $d_\infty$ defined by $d_\infty(\sigma_1,\sigma_2)=\sup_{x\in\Delta^i}\{d(\sigma_1(x),\sigma_2(x))\}$. For every $i\in\mathbb{N}$ we denote by $S_i^{\sm}(X)=C^1(\Delta^i,X)$ the space of \emph{smooth} singular $i$-simplices with values in $X$, i.e. those maps $\Delta^i \to X$ that admit a $C^1$--extension over a neighborhood of $\Delta^i$ in $\R^{i}$. We endow $S_i^{\sm}(X) := C^1(\Delta^i, X)$ with the structure of a measurable space, whose measurable sets are the Borel sets with respect to the $C^1$--topology
  (see \cite[Section 4.2]{LS}). We add the superscript ${}^\sm$ when we consider the smooth version of the complexes above.

As mentioned above, the following result will be proved in Section~\ref{sezione 3}. Then, in Section~\ref{sezione 4} we will deduce our main Theorem~\ref{main theorem}
from Theorem~\ref{pseudostraight:thm}. Finally, in Section~\ref{sezione 5} we will deal with Theorem \ref{productEstimates}.

\begin{theorem}\label{pseudostraight:thm}
Let $N$ be a complete Riemannian manifold with Riemannian universal covering $p_N \: \widetilde{N} \longrightarrow N$, and fix an identification of 
$\pi_1(N)$ with a discrete subgroup $\Lambda<\isom(\widetilde{N})$ such that $N=\widetilde{N}/\Lambda$. Then, there exists a $\Lambda$-equivariant \emph{pseudostraightening operator}
$\st_*\colon S_*^{\lip}(\widetilde{N}) \to S_*^{\sm}(\widetilde{N})$ such that:
\begin{enumerate}
    \item 
    For every $i \in \N$ and $\widetilde \sigma \in S_i^\lip(\widetilde{N})$ there exists a
    $\Lambda$--equivariant preferred Lipschitz homotopy $\widetilde h_{\widetilde \sigma}$ from $ \widetilde \sigma$ to $\st_i(\widetilde \sigma)$ such that, for $1 \le k \le i$,
$$ \widetilde h_{\partial_k \widetilde \sigma} = \widetilde h_{ \widetilde \sigma} \circ (\partial_k \times \id_{[0, 1]}) $$
(where, with a slight abuse of notation, with the symbol $ \partial_k$ we denote both the $k$th face operator and the affine map $\Delta^{n-1} \to \Delta^n$ associated with the $k$th face). 
    \item For every $i \in \N$, there exists a
     function $b_{i} \: \R_{\ge 0} \to \R_{\ge 0}$ such that, for every $ \widetilde \sigma\in S_i^{\lip}(\widetilde{N})$, we have
$$ \lip( \widetilde h_{ \widetilde \sigma}) \le b_i(\lip( \widetilde \sigma)). $$
In particular,
$$ \lip(\st_i(\widetilde \sigma)) \le b_i(\lip(\widetilde \sigma)). $$
    \item For every $0 \le L \in \R$, the family of simplices $\st_i(S_i(\widetilde{N})) \cap S_i^{L}(\widetilde{N})$
    is locally finite.

\item If $\tilde \sigma$ is smooth, the homotopy $\tilde h_{\tilde \sigma}$ is smooth. 
\item The restriction:
$$ S_*^{\sm}(\tilde N) \to C^1(\Delta^i \times I, \tilde N) \qquad \tilde \sigma \mapsto \tilde h_{ \tilde \sigma} $$
is Borel with respect to the $C^1$--topologies. 
\end{enumerate}
\end{theorem}


\section{\label{sezione 3} Pseudostraightening.}


Let $N,\widetilde{N}$ and $\Lambda$ be as at the end of the previous section. We denote by $G$ the group of orientation-preserving isometries of $\widetilde{N}$, with the compact-open topology. 
If $N$ is non-positively curved, then $\widetilde{N}$ is uniquely geodesic. This allows to define a straightening operator $\mathrm{st}_*\colon S_*^\sm(\widetilde{N})\to S_*^\sm(\widetilde{N})$
as follows (see e.g.~\cite{LS}).
Consider a $\Lambda$--invariant set $\widetilde T$ in $\widetilde{N}$, 
and a Borel $\Lambda$--equivariant partition $ \widetilde {\mathscr B}$ of $\widetilde{N}$, in such a way that every element of $\widetilde {\mathscr B}$ contains a single $t \in \widetilde T$. Moreover, suppose that the elements of $ \widetilde {\mathscr B}$ have diameter bounded by $1$, and that every point in $\widetilde N$ is at 
distance at most $1$ from an element in $\widetilde T$. 
Given a $k$--singular smooth simplex $\widetilde \sigma$ in $ \tilde N$, the $k+1$ vertices of $\widetilde{\sigma}$
determine $k+1$ elements of $ \widetilde {\mathscr B}$, whence of $\widetilde T$. Thanks to 
the uniqueness of geodesics of $\widetilde{N}$, 
these points in $\widetilde{T}$ span a well-defined straight simplex $\mathrm{st}_k(\widetilde{\sigma})$.

\subsection*{Convexity and straight homotopies in Riemannian manifolds}
In the general case, a more subtle construction is needed. We say that a subset $A$ of a Riemannian manifold $X$ is \textbf{(geodesically) convex} if for every pair of points $x$ and $y$ in $A$, 
there is a unique minimizing geodesic in $X$ between $x$ and $y$, and this geodesic is contained in $A$. 
Every Riemannian manifold is locally convex, in the sense that for every point of the manifold there is a basis of convex neighbourhoods (see e.g. \cite[Theorem 1.9.10, Corollary 1.9.11]{Kling}). Moreover, a Riemannian manifold is of curvature less than $\kappa$ if and only if it is locally CAT($\kappa$), in the sense that for every point in the manifold there is a convex neighborhood which is a CAT($\kappa$)--space (see e.g. \cite[Theorem 1A.6, Definition 1.2 Chapter II.1]{BriHaf}).

If $x$ and $y$ are sufficiently close points in a Riemannian manifold $X$, we denote by $[x, y] \: [0, 1] \to X$ the unique minimizing constant speed geodesic from $x$ to $y$.

\begin{dfn} \label{straight homotopy}
 Let $\sigma_1$ and $\sigma_2$ be simplices in $S_i(X)$. 
 We say that $\sigma_1$ and $\sigma_2$ are \textbf{sufficiently close} if, for every $x\in\Delta^i$, there exist
 a point $p\in X$ and a positive radius $\rho>0$ (both depending on $x$) such that the ball $B_\rho(p)$ is convex and
 contains both $\sigma_1(x)$ and $\sigma_2(x)$. If $\sigma_1$ and $\sigma_2$ are sufficiently close, 
 then there is a well defined \textbf{straight homotopy} $[\sigma_1, \sigma_2]$ between $\sigma_1$ and $\sigma_2$ given by:
$$ [\sigma_1, \sigma_2](x, t) := [\sigma_1(x), \sigma_2(x)](t). $$ 
\end{dfn}



We will see later that, if $d_\infty(\sigma_1, \sigma_2)$ is sufficiently small, then $\sigma_1$ and $\sigma_2$ are sufficiently close according to our definition and, if $\sigma_1$ and $\sigma_2$ are Lipschitz,
there is a good control of the Lipschitz constant of $[\sigma_1, \sigma_2]$ in terms of the Lipschitz constants of $\sigma_1$ and $\sigma_2$. Moreover, since the exponential map is a local diffeomorphism, it follows that the homotopy is smooth, if $\sigma_1$ and $\sigma_2$ are smooth. Since we will need to compose straight homotopies in order to obtain a smooth homotopy, we define a version of the straight homotopy which is constant near $0$ and $1$: 
\begin{equation} \label{o378yqirhus} [\sigma_1, \sigma_2]_\sm(x, t) := [\sigma_1, \sigma_2](x, \xi(t)), \end{equation}
where $\xi \: [0, 1] \to [0, 1]$ is a fixed smooth surjective non-decreasing map, locally constant in a neighborhood of $\{0,\,1\}$.


We are now ready to begin the proof of Theorem~\ref{pseudostraight:thm}.

\begin{dfn} \label{ } For every $L>0$, we define a map $ \mathfrak r_L \colon S_*^{L}(\widetilde N) \to \R_{\ge 0}$ as follows. For 
$\sigma \in S_*^{L}(\widetilde{N})$, the value $\mathfrak r_L(\sigma)$ is equal to the supremum of the set of real numbers $r\in\R$ which satisfy the following property: if $\sigma_1$ and $\sigma_2 \in S_i^{L}(\widetilde{N})$ and $d_\infty(\sigma_i, \sigma) \le r$, then 
$\sigma_1$ and $\sigma_2$ are sufficiently close, and
the straight homotopy $[\sigma_1, \sigma_2]$ satisfies $\lip([\sigma_1, \sigma_2]) \le 4L$.
\end{dfn}


\begin{lem} \label{ohnfjcklw} For every $L \ge 1$ the function:
$$ \mathfrak r_L \: (S_i^{L}(\widetilde{N}), \, d_\infty) \to \R $$
is $G$--invariant, strictly positive, and locally $1$--Lipschitz, hence is Borel with respect to the $C^0$--topology. It follows in particular that the restriction of $\mathfrak r_L$ on $S_i^{L, \sm}(\widetilde{N})$ is Borel with respect to the $C^1$--topology.
\end{lem}

\begin{proof} The $G$--invariance property is immediate. The sectional curvature over a fixed relatively compact (i.e. bounded) set is bounded above by a constant $\kappa \in \R$ (which depends on the relatively compact set). 
Given $\sigma \in S_i^{L}(\widetilde{N})$, let $V$ be a bounded neighborhood of $\sigma(\Delta^i)$, and let $0 < \kappa \in \R$ be such that the curvature on $V$ is less than $\kappa$. 
This means that $V$ admits a cover by CAT($\kappa$) convex open subsets, which induces in turn an open cover of $\sigma(\Delta^i)$. Since $\sigma(\Delta^i)$ is compact, such a cover admits a Lebesgue number $1>\rho>0$. Therefore, up to decreasing $\rho$, we have that
for every 
$y \in \sigma(\Delta^i)$, the ball $B_\rho(y) \subset V$ is convex and CAT$(\kappa)$, of diameter at most $1/10$ of the diameter of the ``comparison sphere'' $S_{\kappa}$ of constant curvature $\kappa$.



So, let $\sigma_1$ and $\sigma_2$ be Lipschitz $i$--simplices with $\lip(\sigma_j) \le L$ and such that $d_\infty(\sigma_j, \sigma) \le \rho/2$, $j=1,\,2$. By construction, $\sigma_1$ and $\sigma_2$ are sufficiently close, so there exists a well-defined straight homotopy $[\sigma_1,\sigma_2]$. Since $\Delta^i\times [0,1]$ is a geodesic space, in order to show that $[\sigma_1,\sigma_2]$ is $4L$-Lipschitz, it is sufficient to show that it is \emph{locally} $4L$-Lipschitz. This implies at once that
$\mathfrak r_L(\sigma)\geq \rho /2>0$. 

We fix a point $(x,t)\in \Delta^i\times [0,1]$. Therefore, 
we may restrict to consider points $(y,s)\in\Delta^i\times [0,1]$ such that 
$\sigma_j(x)$ and $\sigma_j(y)$ are contained in $B_\rho(\sigma(x))$ for $j=1,\,2$. For simplicity, we set $a=[\sigma_1(x), \sigma_2(x)](t)$, $b=[\sigma_1(y), \sigma_2(y)](s)$, $c=[\sigma_1(x), \sigma_2(y)](t)$, $b'=[\sigma_1(y), \sigma_2(y)](t)$ (see Figure~\ref{}). We now need to prove that $d(a,b)\leq 4L\cdot d((x,t),(y,s))$.
\begin{center}
\begin{figure}
\includegraphics[width=6cm]{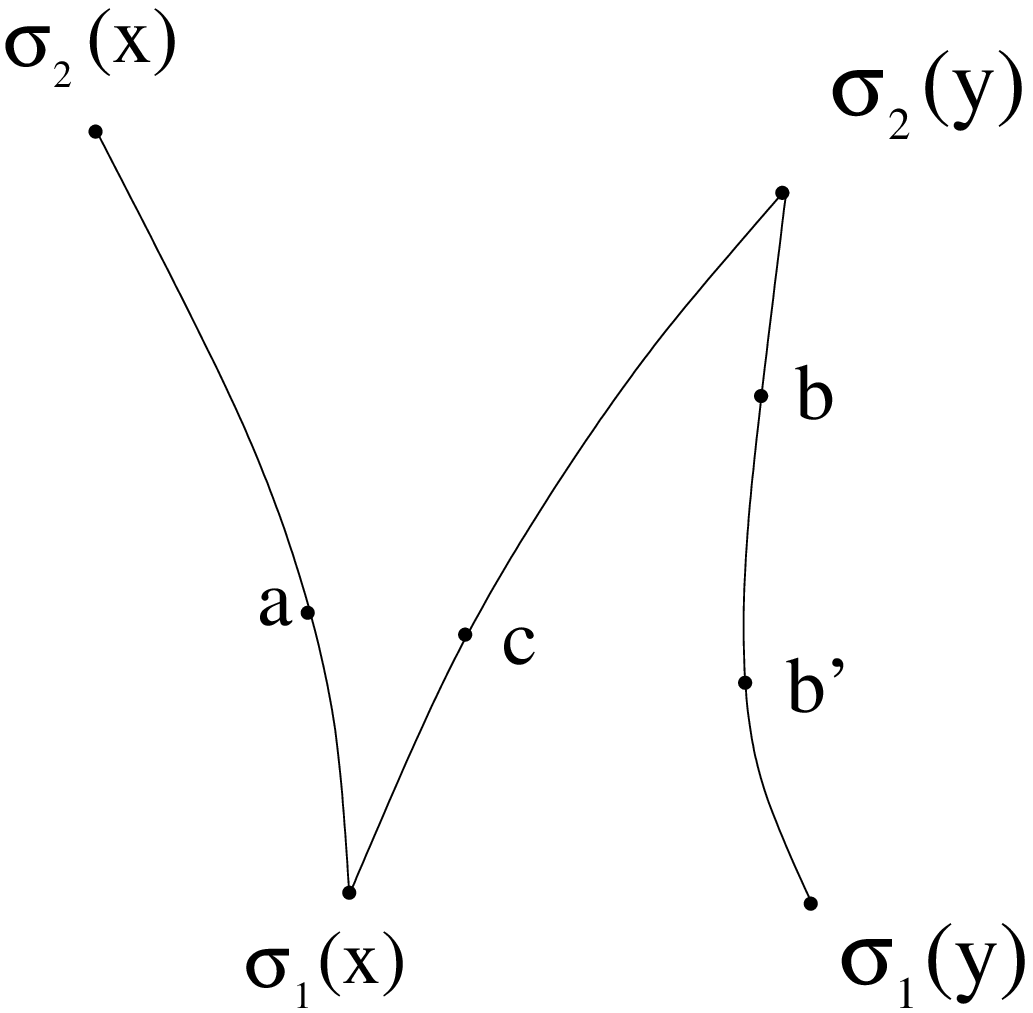}
\end{figure}
\end{center}
We have
\begin{equation}\label{stimaccia} 
d(a,b)\leq d(a,c)+d(c,b')+d(b',b)\ .
\end{equation}
Let us estimate the first term on the right side of the inequality above. 
The points $\sigma_1(x)$, $\sigma_2(x)$ and $\sigma_2(y)$ are the vertices of a triangle in the CAT$(\kappa)$ space $B_\rho(\sigma(x))$. 
For every point $p$ on the perimeter of the triangle with vertices $\sigma_1(x)$, $\sigma_2(x)$, $\sigma_2(y)$, we denote 
by $ \overline p$ the corresponding point in the comparison triangle in $S_\kappa$. 
The perimeter of this triangle is less than one half of the diameter of the comparison sphere $S_\kappa$ so, if we denote by
$d_c$ be the metric on $S_\kappa$, we have:
$$
d(a,c)\leq d_c(\overline{a},\overline{c})\leq d_c(\overline{\sigma_2(x)},\overline{\sigma_2(y)})=
d(\sigma_2(x),\sigma_2(y))\leq L\cdot d(x,y)\ ,
$$
where the first inequality is just the CAT$(k)$ inequality, and the second one follows from the fact that
the comparison triangle of vertices $ \overline{\sigma_1(x)}$, $ \overline{\sigma_2(x)}$, 
$ \overline{\sigma_2(y)}$ is  contained in a single hemisphere of the comparison sphere.
The very same argument applied to the triangle with vertices $\sigma_2(y)$, $\sigma_1(x)$ and $\sigma_1(y)$
shows that $d(c,b')\leq L\cdot d(x,y)$, while of course we have 
$d(b',b)=|t-s|d(\sigma_1(y),\sigma_2(y))\leq |t-s|d_\infty(\sigma_1,\sigma_2)\leq |t-s|\rho$.
Putting these inequalities together with~\eqref{stimaccia} we finally get
\begin{align*}
d(a,c)&\leq 
2Ld(x, y) + d_\infty(\sigma_1, \sigma_2) |t-s| \le 2L d(x, y) + \rho |t-s| \leq 
2L (d(x, y) + \rho |t-s|) \\ & \leq 2 \sqrt 2 L \sqrt{d(x, y)^2 + |t-s|^2} \leq 4L d((x, t),\, (y, s)), 
\end{align*}
where we used the fact that $\alpha + \beta \le \sqrt 2 \sqrt{\alpha^2 + \beta^2}$
for every $\alpha,\beta\in\R$.
We have thus proved that $\mathfrak r_L(\sigma)\geq \rho/2>0$. 

We now prove our claim about the Lipschitz constant of $ \mathfrak  r_L$. Let $\sigma \ne \tau \in S_*^{L}(\widetilde{N})$ be such that $d_\infty(\tau, \sigma) = r < \mathfrak r_L(\sigma)$. Let $\sigma_1,\,\sigma_2 \in 
S_i^{L}(\widetilde{N}) \cap B_{\mathfrak r_L(\sigma) - r}(\tau) \subseteq B_{\mathfrak r_L(\sigma)}(\sigma)$. 
By the last inclusion, it follows that $\lip [\sigma_1, \sigma_2] \le 4L$ hence, by definition of $ \mathfrak r_L$, $\mathfrak r_L(\tau) \ge \mathfrak r_L(\sigma) - r$. Put now $r \le \mathfrak r_L(\sigma) / 2$. 
It follows that $d_\infty(\tau, \sigma) = r \le \mathfrak r_L(\sigma) - r \le \mathfrak r_L(\tau)$. Hence $d_\infty(\sigma, \tau) \le \min\{\mathfrak r_L(\sigma),\, \mathfrak r_L(\tau)\}$.
 Therefore $ \mathfrak r_L(\tau) + r \ge \mathfrak r_L(\sigma) \ge \mathfrak r_L(\tau) - r$, hence $|\mathfrak r_L(\sigma) - \mathfrak r_L(\tau)| \le r$.


\end{proof}

We define a bounded function 
$$ \mathfrak r \: S_i^{\lip}(\widetilde N) \to \R \qquad \mathfrak r(\sigma) = \min \left\{\mathfrak r_{\lceil \lip (\sigma) \rceil +1}(\sigma),\, 1\right\} $$


Since $S_i^L(\widetilde{N})$ is closed in $S_i(\widetilde N)$ for every $L \in \N$, the sets $S_i^{L+1}(\widetilde{N}) \setminus S_i^L(\widetilde{N})$ 
are Borel in $S_i^\lip(\widetilde N)$ with respect to the $C^0$--topology. Therefore, the restriction of $ \mathfrak r$ to $S_i^\sm(\widetilde N)$ is Borel with respect to the $C^1$--topology.

\begin{dfn} \label{ } Given two subsets $Y$ and $Z$ of $S_i^\lip(\widetilde N)$, we say that $Y$ is \textbf{sufficiently dense} in $Z$ 
if for every $\sigma \in Z$ there is $ \overline \sigma \in Y \cap Z$ such that $d_\infty( \sigma, \overline \sigma) \le \mathfrak r(\sigma)$. 
\end{dfn} 

\begin{rem} \label{} If $Y$ is sufficiently dense in $Z$ we have that, for every $\sigma \in Z$, there is $ \overline \sigma \in Y \cap Z$ such that, if $ \lceil \lip (\sigma) \rceil + 1 \ge \lip (\overline \sigma)$, the straight homotopy $[\sigma, \overline \sigma]$ is $4$$(\lceil \lip \sigma \rceil +1$)--Lipschitz. 
\end{rem}

\begin{rem} \label{346t7eiywgudhs} Let $Z$ be a relatively compact subset of $S_i^\lip( \tilde N)$ (with respect to the $d_\infty$--metric). From the continuity of $ \mathfrak r_L$, $1 \le L \in \N$, it follows that $ \mathfrak r$ has positive infimum on $Z$. Since $Z$ is totally bounded, it admits a finite sufficiently dense subset $Y \subseteq Z$. 
%
\end{rem}




\subsection*{The 0-dimensional case}
Our proof of Theorem~\ref{pseudostraight:thm} constructs the map $\st_*$ inductively on the dimension of the simplices.
For technical reasons, we will prove a slightly more precise version of Theorem~\ref{pseudostraight:thm}, where in point (1) we require that
the homotopy $h_\sigma$ between the $i$-dimensional simplex $\sigma$ and $\st_i(\sigma)$ is constant on the subintervals
$[0,2^{-i-2}]\subseteq [0,1]$ and $[1-2^{-i-2},1]\subseteq [0,1]$.

Let $ \mathscr P = \{P_\alpha\}_{\alpha \in \N}$ be a locally finite partition of evenly covered Borel sets of $N$ whose diameter is bounded above by $1$, and fix a lift $ \widetilde P_\alpha \subseteq \widetilde N$ of $P_\alpha$ for every $\alpha$. In particular the diameter of $\widetilde P_\alpha$ is at most $2$. 
The union $\widetilde{F}=\bigcup_{\alpha \in \N} \widetilde P_\alpha$ is a Borel fundamental domain for the action of $\Lambda$ on $\widetilde N$.

Since every $\widetilde{P}_\alpha$ is relatively compact, by Remark \ref{346t7eiywgudhs} we have that $S_0(\widetilde{P}_\alpha)=\widetilde{P}_\alpha$ contains a finite sufficiently dense subset $\widehat{S}_0(\widetilde{P}_\alpha)$. 

We now define $\widehat{S}_0(\widetilde{F})$ as $\bigcup_{\alpha \in \N} \widehat{S}_0(\widetilde{P}_\alpha)$. For every $\overline{\sigma}\in \widehat{S}_0(\widetilde{F})$ we consider the set
$$
A_{\overline{\sigma}}'=\left\{ \sigma\in S_0(\widetilde{F})\, |\, d_\infty(\sigma,\overline{\sigma})\leq \mathfrak{r}( \sigma)\right\}\ .
$$

By definition, the $A'_{\overline{\sigma}}$ provide a cover of $S_0(\widetilde{F})$ by closed subsets. We may now order the
$A'_{\overline{\sigma}}$'s via a bijection with the natural numbers,
 and set $A_{\overline{\sigma}}=A'_{\overline{\sigma}}\setminus\bigcup_{\overline{\tau}<\overline{\sigma}} A'_{\overline{\tau}}$. In this way
we get a partition (possibly containing some empty subsets)  of
$S_0(\widetilde{F})$ into \emph{Borel} subsets (with respect to the $C^0$--topology). Finally, for every $\sigma\in S_0(\widetilde{F})$ we set
$$
\st_0(\sigma)=\overline{\sigma}\quad \textrm{if\ and\ only\ if}\quad \sigma\in A_{\overline{\sigma}}
$$
and extend the definition of $\st_0$ over the whole of $S_0(\widetilde{N})$ by $\Lambda$-equivariance. For every $\sigma\in S_0(\widetilde{N})$, we also set $h_\sigma$
to be the concatenation of the constant paths on $[0,1/4]$ and on $[3/4,0]$ with the constant speed parameterization of the 
path
$[\sigma,\st_0(\sigma)]_\sm$ on $[1/4,3/4]$.

Let us check that $\st_0$ satisfies the four conditions required in the statement of Theorem~\ref{pseudostraight:thm}. 
The fact that $\st_0$ and the map $\sigma\mapsto h_\sigma$ are equivariant is obvious. Moreover, since
$d_\infty(\sigma,\st_0(\sigma))\leq 1$ for every $\sigma$, the homotopy $h_\sigma$ is always $2\lip (\xi)$--Lipschitz (see \eqref{o378yqirhus}). The family
$\widehat{S}_0(\widetilde{N})$ is locally finite by construction. Finally, the map $\st_0$ is clearly Borel, so the map
$\sigma\mapsto h_\sigma$ is Borel as well, being continuous on every Borel set $A_{\overline{\sigma}}$ defined above.

\subsection*{The inductive step}
We now suppose that a map $\st_k$ satisfying the properties stated in Theorem~\ref{pseudostraight:thm}
has been constructed for every $k\leq i$, and proceed
with the inductive step. We need two technical lemmas: 

\begin{lem} \label{the lemma below} Let $M$ be a smooth manifold. 
\begin{enumerate}
    \item Let $f \: \partial \Delta^i \to M$ be such that its restriction on every face of $\partial \Delta^i$ is smooth. Then $f$ is smooth.

    \item Let $F \: \partial \Delta^i \times I \to M$ be such that $F \circ (\partial_j \times \id_I)$ is smooth for every $0 \le j \le i$. Then $F$ is smooth.
\end{enumerate}
\end{lem}

\begin{proof} See \cite[Lemma 16.8, p. 420]{Lee} for a proof of (1). Up to diffeomorphism, we may see $\partial \Delta^i \times I$ as a truncated $\Delta^{i+1}$. Hence (2) follows as well.
\end{proof}

\begin{lem} \label{approssimazione con roba liscia} For every $L > 0$ and every smooth map $s \: \partial \Delta^i \to N$ the following inclusion is  dense (with respect to the $d_\infty$--metric):
$$ \{\sigma \in S_i^{\sm}(N) \: \sigma {}_{\big| \partial \Delta^i} = s, \lip(\sigma) < L \} \hookrightarrow \{\sigma \in S_i^{\lip}(N) \: \sigma {}_{\big| \partial \Delta^i} = s, \lip(\sigma) < L \}. $$
\end{lem}

\begin{proof} We first prove the theorem in the case $N = \R^n$. Let $\sigma \: \Delta^i \to \R^n$ be an $L$--Lipschitz simplex and let $ \varepsilon> 0$. We extend $\sigma$ to an $L$--Lipschitz map over the whole $\R^i$ by composing it on the right with a retraction $\R^i \to \Delta^i$. Then, using convolution, we find a smooth $L$--Lipschitz simplex $\sigma' \: \Delta^i \to \R^n$ such that $d_\infty(\sigma, \sigma') < \varepsilon$. 

Now, suppose that $\sigma {}_{\big| \partial \Delta^i} \: \partial \Delta^i \to \R^n$ is smooth. Let $\xi \: [0 , 1] \to [0, 1]$ be a smooth non-decreasing surjective map, locally constant on $\{0, 1\}$, such that $\lip \xi$ is close to $1$, and $\|\xi - \id\|_\infty$ is small. We suppose that the barycenter of $\Delta^i$ is the origin of $\R^i$. Then we define a smooth simplex $ \overline \sigma$ as follows:

$$
\overline \sigma(x) = \left\{
\begin{array}{ll} 
[\sigma(x), \sigma'(x)](\xi(\frac{s}{\varepsilon} + \frac{ \varepsilon - 1}{\varepsilon})) \hfill \textrm{ if } x = sy,\, y \in \partial \Delta^i,\, s \in [1- \varepsilon, 1] \\ 
\sigma'(\xi( \frac{s y}{1- \varepsilon})) \hfill \textrm{ if } x = sy,\, y \in \partial \Delta^i,\, s \in [0, 1- \varepsilon] \\ 
\end{array} 
\right. 
$$
For $\|\xi - \id_{[0,\,1]}\|_\infty$ sufficiently small, and $\lip(\xi)$ sufficiently close to $1$, it is easily seen that $\|\overline \sigma - \sigma\|_\infty \le 2 \varepsilon$ and $\lip( \overline \sigma) \le \lip (\sigma) + \varepsilon$. 

Now we consider the general case. By the Nash embedding theorem, there is $n \in \N$ and a smooth isometric embedding $i \: N \hookrightarrow \R^n$. This means that the length of a path $\gamma$ in $N$ coincides with the length of the path $i \circ \gamma$ in $\R^n$. 

We identify $N$ and $i(N)$, and we consider a tubular neighborhood $U$ of $N$ in $\R^n$, together with a smooth retraction $r \: U \to N$. Let $\sigma \: \Delta^i \to N \hookrightarrow \R^n$ be a Lipschitz simplex with smooth boundary and $\lip(\sigma) \leq L - 2\delta$, for some $\delta > 0$. Fix $ \varepsilon > 0$. Up to restricting the tubular neighborhood $U$ of $N$, we can assume that $ \lip (r) < \frac{L}{L - \delta}$, and that $d_\infty(x, r(x)) < \varepsilon/2$ for every $x \in U$. Let $ \overline \sigma \: \Delta^i \to U \subseteq \R^n$ be a smooth simplex such that $\lip ( \overline \sigma) \leq L - \delta$, $d_\infty(\sigma, \overline \sigma) < \varepsilon/2$ and $ \overline \sigma {}_{\big| \partial \Delta^i} = \sigma {}_{\big| \partial \Delta^i}$. Then, $r \circ \overline \sigma \: \Delta^i \to N$ is a smooth simplex, with $\lip (r \circ \overline \sigma) \le \lip(r) \lip( \overline \sigma) < \frac{L}{L - \delta} (L - \delta) = L$. Moreover, $\sigma {}_{\big| \partial \Delta^i} = r \circ \overline \sigma {}_{\big| \partial \Delta^i}$. Finally:
$$ d_\infty(\sigma, r \circ \overline \sigma) \le d_\infty(\sigma, \overline \sigma) + d_\infty( \overline \sigma, r \circ \overline \sigma) \le \varepsilon/2 + \varepsilon/2 = \varepsilon. $$
\end{proof}

Let $\sigma$ be an element of $S_{i+1}^\lip(\widetilde{N})$. We first homotope $\sigma$ into an $(i+1)$-simplex
with pseudostraight faces as follows. For every $j=0,\ldots, i+1$, associated to the face $\partial_j\sigma$ there are the pseudostraight
$i$-simplex $\st_{i}(\partial_j\sigma)$ and the homotopy $h_{\partial_j \sigma}$.
We can think, up to a translation in $\R^n$, that the barycenter of the standard $(i+1)$-simplex $\Delta^{i+1}$ 
is $0 \in \R^{i+1}$. Let $\Delta' := \frac{1}{2} \Delta^{i+1}$. 
We know by induction that the homotopies $h_{\partial_j\sigma}$ and $h_{\partial_k\sigma}$ coincide on $(\partial_j\Delta^{i+1})\cap (\partial_k\Delta^{i+1})$, so
a global homotopy
$h_{ \partial \sigma}\colon \partial\Delta^{i+1}\times I\to \widetilde{N}$ is defined. 
Moreover, we know by induction that this homotopy is constant on the time intervals $[0, \,2^{-i-2}]$ and $[1 - 2^{-i-2},\, 1]$. 

We define a homotopy $h'_\sigma$ as follows: \begin{enumerate} 
	\item[$(H1)$] In the time interval $[0,\, 2^{-i-3}]$ the homotopy is constantly $\sigma$.
	\item[$(H2)$] In the time interval $[2^{-i-3}, \,2^{-i-2}]$, we construct $h_\sigma'$ by composing $\sigma$ with $H$ on the left: $$ \Delta^{i+1} \times [2^{-i-3}, \,2^{-i-2}] \overset{H}{\longrightarrow } \Delta^{i+1} \overset{\sigma}{\longrightarrow} \widetilde N, $$ where $H$ is 
a smooth (rescaled) homotopy between the identity and a map which sends $\Delta'$ onto $\Delta$ and retracts $ \Delta \setminus \Delta'$ onto $ \partial \Delta$, and which is locally the constant homotopy for $t$ in a neighborhood of $\{2^{-i-3}, \,2^{-i-2}\}$. The Lipschitz constant of the homotopy $\sigma \circ H$ is at most $\lip(H) \cdot \lip (\sigma)$. The boundary of $\sigma$ has been left unchanged so far.
    \item[$(H3)$] Now we consider the time interval $[2^{-i-2},\, 1 - 2^{-i-2}]$. On the boundary $ \partial \Delta = 2 \partial \Delta'$ we have a preassigned homotopy $h_{ \partial \sigma}$. Let $s \: \overline{\Delta \setminus \Delta'} \to [1/2, 1]$ be such that $x \in s \partial \Delta$, for $x \in \overline{\Delta \setminus \Delta'}$, and let $\hat s \: \overline{\Delta \setminus \Delta'} \to [1/2, 1]$ be the composition of $s$ with a non-decreasing smooth surjective map: $[1/2, 1] \to [1/2, 1]$ which is constant on a neighborhood of $\{1/2,\,1\}$. We define $h_\sigma'$ in the time interval $[2^{-i-2},\, 1 - 2^{-i-2}]$ by:
$$ \overline{\Delta \setminus \Delta'} \times [2^{-i-2},\, 1 - 2^{-i-2}] \ni (x, t) \mapsto h_{ \partial \sigma} \Big(x/ s(x), \frac{(2 \hat s(x)-1)}{1 - 2^{-i-1}} t\Big), $$ 
constant on $\Delta'$. This homotopy is the composition of $h_{\partial \sigma}$ and a Lipschitz map: $\overline{\Delta \setminus \Delta'} \times I \to \overline{\Delta \setminus \Delta'} \times I$. Therefore its Lipschitz constant is uniformly bounded, up to a multiplicative constant, only by the Lipschitz constant of $h_{\partial \sigma}$, which in turn is uniformly bounded by the Lipschitz constant of $\sigma$ by induction. Similarly, it is proven by induction and by Lemma \ref{the lemma below} that the homotopy is smooth, if every face of $\sigma$ is smooth.
\end{enumerate}

Let us now summarize what we have obtained so far. For any given $(i+1)$-simplex $\sigma$ we have produced a homotopy $h'_\sigma \: \Delta^{i+1} \times I \to \tilde N$ between
$\sigma$ and an $(i+1)$-simplex $\sigma'$ with pseudostraight faces. Our choice for $h'_\sigma$ is obviously $\Lambda$-equivariant. 
Moreover, by point (H1), (H2) and (H3) above, the Lipschitz constant of $h'_\sigma$ is bounded in terms of the Lipschitz constant of $\sigma$, and $h_\sigma'$ is smooth if $\sigma$ is smooth. 

We now need to approximate $\sigma'$
with a suitably chosen pseudostraight $(i+1)$-dimensional simplex. In order to do so, we begin with the following lemma. We denote by $\widehat{S}_i(\widetilde{N})$
the set $\st_i(S_i^\lip(\widetilde{N})) \subset S_i^\sm(\widetilde{N})$, and we set $\widehat{S}_i^L(\widetilde{N}) = \widehat{S}_i(\widetilde{N}) \cap S_i^L(\widetilde{N})$ for every $L\in\R$.
We will concentrate our attention on simplices whose faces are pseudostraight, so we reserve a symbol for this set:
$$
B\widehat{S}_{i+1}(\widetilde{N}):=\{\sigma\in S_{i+1}^\lip(\widetilde{N})\, :\, \partial_j\sigma\in \widehat{S}_i(\widetilde{N})\ \textrm{for\ every}\ j\}\ .
$$
Moreover, if $\tau\in B\widehat{S}_{i+1}(\widetilde{N})$, then we set 
$$ \mathscr{B}_{\tau} := \left\{\tau'\in S_{i+1}^\lip(\widetilde{N}) \: \tau'{}_{|\partial \Delta} = \tau {}_{| \partial \Delta} \right\} \ . $$

\begin{lem} \label{26twb45} 
There exists a $\Lambda$--invariant family $\widehat{S}_{i+1}(\widetilde{N})\subset S_{i+1}^\sm(\widetilde{N}) \subset S_{i+1}^\lip(\widetilde{N})$ which satisfies the following properties: 
\begin{enumerate}
	\item for every $L \in \N$, the family $\widehat{S}_{i+1}^L(\widetilde{N})=\widehat{S}_{i+1}(\widetilde{N})\cap S_{i+1}^L(\widetilde{N})$ 
	is locally finite.
\item If $\tau\in B\widehat{S}_{i+1}(\widetilde{N})$,
then $\widehat{S}_{i+1}^{L}(\widetilde{N})$ is sufficiently dense in $\mathscr{B}_{\tau} \cap S_{i+1}^{<L}(\widetilde{N})$ for every $L\in\N$, where $S_{i+1}^{<L}(\widetilde{N}) := \left\{\sigma \in S_{i+1}(\widetilde{N}) \: \lip \sigma < L \right\}$.
\end{enumerate} 
\end{lem}

\begin{proof} Let $ \{\widetilde{P}_\alpha\}_{\alpha\in \N}$ be the locally finite family of Borel subsets of $\widetilde{N}$ 
fixed above, so that $\widetilde{F}=\bigcup_{\alpha\in\N} \widetilde{P}_\alpha$ is a Borel fundamental domain for the action
of $\Lambda$ on $\widetilde{N}$.

Let $L \in \N$. We define a subset $ \mathfrak P_\alpha^{[L-1, L)}\subseteq B\widehat{S}_{i+1}^\lip(\widetilde{N})$ as follows:
$\sigma\in \mathfrak P_\alpha^{[L-1, L)}$ if and only if $\lip(\sigma)\in [L-1, L)$ and the first vertex of $\sigma$ lies in 
$ \widetilde P_\alpha$.
This set is totally bounded with respect to the metric $d_\infty$. 
Indeed, $ \mathfrak P_\alpha^{[L-1, L)}$ is equicontinuous (because $L$ is fixed) and pointwise relatively compact 
(because the images of simplices in $ \mathfrak P_\alpha^{[L-1, L)}$ are uniformly bounded, and bounded subsets of $\widetilde{N}$ are relatively compact). 
By Arzelà--Ascoli's theorem, $\mathfrak P_\alpha^{[L-1, L)}$ is relatively compact in $S_{i+1}(\widetilde N)$, hence totally bounded. 
It follows easily from the local finiteness of 
$\widehat{S}_{i}^L(\widetilde{N})$ 
that the \emph{geometric boundary} of $\mathfrak P_\alpha^{[L-1, L)}$:
$$ \partial_g (\mathfrak P_\alpha^{[L-1, L)}) := \left\{\sigma {}_{| \partial \Delta} \: \sigma \in \mathfrak P_\alpha^{[L-1, L)} \right\} 
$$
is finite. For every $s\in  \partial_g (\mathfrak P_\alpha^{[L-1, L)})$ the set of elements of $\mathfrak P_\alpha^{[L-1, L)}$ whose restriction
to $\partial \Delta^{i+1}$ coincides with $s$ is totally bounded, so it admits a finite sufficiently dense subset $D_s^{[L-1, L)}$. By Lemma \ref{approssimazione con roba liscia}, we can require that the elements in $D_s^{[L-1, L)}$ are smooth. By taking the union of the $D_s^{[L-1, L)}$'s over
$s\in  \partial_g (\mathfrak P_\alpha^{[L-1, L)})$ we thus get a finite set $D_\alpha^{[L-1, L)}$ of simplices in $\mathfrak P_\alpha^{[L-1, L)}$, which approximates nicely
all the elements of $\mathfrak P_\alpha^{[L-1, L)}$.
We put
$$ \widehat{S}_{i+1} (\widetilde{N}):= 
\bigsqcup_{\lambda \in \Lambda, \, L \in \N, \, \alpha \in \N} \lambda\cdot  D^{[L-1, L)}_\alpha\ . $$

%
It is easy to check that both conditions of the lemma are satisfied by construction.
\end{proof} 
For every 
$\overline{\sigma}\in \widehat{S}_{i+1}(\widetilde{N})$ we set 
$$
A'_{\overline{\sigma}}=\left\{\sigma \in S_{i+1}^\lip(\widetilde{N}) \, :\, \sigma|_{\partial\Delta^{i+1}}=\overline{\sigma}|_{\partial\Delta^{i+1}},\, \lip(\sigma)   \geq \lfloor \lip(\overline{\sigma}) \rfloor\, , 
d_\infty(\sigma,\overline{\sigma})\leq \mathfrak{r}(\sigma) \right\}\ .
$$

By the previous lemma, the $A'_{\overline{\sigma}}$'s cover $B\widehat{S}_{i+1}^\lip(\widetilde{N})$. Being locally finite, the family $\widehat{S}_{i+1}(\widetilde{N})$ is countable. Moreover, the map $\sigma \mapsto \lip (\sigma)$ is a Borel function on $S_{i+1}^\lip(\widetilde{N})$ (with respect to the $C^0$--topology), and this readily implies that each $A'_{\overline{\sigma}}$
is Borel. Therefore, just as we did in the 0-dimensional case, we can choose Borel subsets $A_{\overline{\sigma}}\subseteq A'_{\overline{\sigma}}$
which provide a partition of $B\widehat{S}_{i+1}^\lip(\widetilde{N})$ (possibly with some empty sets). 

We are now ready to define the map $\st_{i+1}$. 
So, let us take $\sigma\in S_{i+1}^\lip(\widetilde{N})$. We have already associated to $\sigma$
a simplex $\sigma'\in B\widehat{S}_{i+1}^\lip(\widetilde{N})$.
Also recall that a homotopy $h'_\sigma$ between $\sigma$ and $\sigma'$ has also been constructed in such a way that
$\lip(h'_\sigma)\leq c_{i+1}(\lip(\sigma))$ for some fuction $c_{i+1}\colon [0,\infty)\to [0,\infty)$
(so, in particular, the inequality $\lip(\sigma')\leq c_{i+1}(\lip(\sigma))$ also holds). We now define $\st_{i+1}(\sigma)$ as follows:
$$
\pst_{i+1}(\sigma)=\overline{\sigma}\quad \textrm{if\ and\ only\ if}\ \sigma'\in A_{\overline{\sigma}},
$$
and complete the homotopy given in points (H1), (H2), (H3) above with two more steps:

\begin{enumerate}
    \item[$(H4)$] In the time interval: $[1 - 2^{-i-2}, 1- 2^{-i-3}]$ we do the homotopy between $\sigma'$ and $\pst_{i+1}(\sigma)$ by sending $(x, t) \in \Delta^{i+1} \times [1 - 2^{-i-2}, 1- 2^{-i-3}]$ to
$$ [\sigma'(x), (\pst_{i+1}(\sigma))(x)]_\sm(2^{i+3}t + 2 - 2^{-i-3}). $$
    \item[$(H5)$] Finally, the homotopy is constant in the time interval $[1 - 2^{-i-3}, 1]$.
\end{enumerate}

The homotopy $h_\sigma$ just described is easily seen to fulfill the Lipschitz condition described in point (2) of Theorem \ref{pseudostraight:thm}, for some function $b_{i+1}$. The distance $d_\infty(\sigma,\, \sp_{i+1}(\sigma))$ is bounded by a function which depends only on $\lip(\sigma)$ and $i+1$. Indeed, it is easily seen that $d_\infty(\sigma,\, \sigma') \le \diam (\sigma) \le L$ 
and $d_\infty(\sigma',\,\sp_{i+1}(\sigma)) \le 1$. This implies in particular that $\sp_{i+1}$ maps locally finite families of simplices to locally finite ones.

Finally, the assignment $S_{i+1}^\sm(\tilde N) \ni \sigma \mapsto h_\sigma \in C^1(\Delta\times I, \tilde N)$ is Borel (with respect to the $C^1$--topologies). First we prove that $\sp_{i+1} \: S_{i+1}^\sm(\tilde N) \to S_{i+1}^\sm(\tilde N)$ is Borel with respect to the $C^1$--topologies. Since the image $\sp_{i+1}(S_{i+1}^\sm(\tilde N))$ is countable, it is sufficient to prove by induction that the map $\sp_{i+1} \: S_{i+1}(\tilde N) \to S_{i+1}(\tilde N)$ is $C^0$--Borel. The map $\sp_{i+1}$ is the composition of the map $\sigma \mapsto \sigma'$ described in points (H1), (H2), (H3) and $\sigma' \mapsto \sp_{i+1}(\sigma)$ given in (H4) and (H5). The second map is $C^0$--Borel by construction (because the $A_{ \overline \sigma}' \subset S_{i+1}(\tilde N)$ are $C^0$--Borel). The first one is $C^0$--Borel because it is continuous on every element of a countable partition of Borel sets. Indeed, fix $ \overline \sigma \in \widehat{S}_i(\widetilde N)$, and consider the set:
$$ C_{\overline \sigma {}_{\big| \partial \Delta}} := \{\sigma \: \sigma' {}_{| \partial \Delta} = \overline \sigma {}_{| \partial \Delta} \} = \{\sigma \: \sp_{i-1}(\partial_k\sigma) = \partial_k\overline \sigma \,\,\,\, \forall \, 0 \,\, \le k \le i\}. $$

This is a $C^0$--Borel set because the boundary operators are Borel and $\sp_{i}$ is Borel by induction. The restriction of the assignment $\sigma \mapsto \sigma'$ over $C_{ \overline \sigma}$ is continuous.

Hence $\sp_{i+1}$ is Borel, and in particular the sets $\sp_{i+1}^{-1}(\overline \sigma) \cap S_{i+1}^\sm(\tilde N)$ are $C^1$--Borel. To prove that the map $\sigma \mapsto h_{\sigma}$ is $C^1$--Borel one simply notes that its restriction on every $\sp_{i+1}^{-1}(\overline \sigma) \cap S_{i+1}^\sm(\tilde N)$ is continuous with respect to the $C^1$--topologies.

TEOREMA COMMENTATO 10 OTTOBRE

TEOREMA COMMENTATO 10 OTTOBRE


\subsection*{Straight chains compute the Lipschitz simplicial volume}
We now give some consequences of Theorem \ref{pseudostraight:thm}. In the following Lemma, by $C_*^\lip(X)$ we mean the subcomplex of $C_*(X)$ whose chains are (finite) sums of Lipschitz simplices.

\begin{lem} \label{da chain top a chain homo pag 11 LS} 
%
Let $X$ be a Riemannian manifold, and let be given Lipschitz homotopies $ \widetilde h_{ \widetilde \sigma}$ for every simplex $ \widetilde \sigma \in S_*^\lip(X)$. Put: $f^{(m)}(\widetilde \sigma) = \widetilde h_{ \widetilde \sigma} {}_{| \Delta^* \times \{m\}}$, $m \in \left\{0, 1\right\}$. Suppose that

\begin{enumerate}
    \item $ \widetilde h_{\partial_k \widetilde \sigma} = \widetilde h_{ \widetilde \sigma} \circ (\partial_k \times \id_I)$ for every Lipschitz simplex $ \widetilde \sigma$.
\end{enumerate}    
Then the induced maps:
$$ f_*^{(m)} \: C_*^\lip(X) \to C_*^\lip(X) \qquad m = 0, \, 1 $$ 
are chain-homotopic. Moreover, suppose that:
\begin{enumerate}    
    \item[(2)] there are functions $b_i \: \R_{\ge 0} \to \R_{\ge 0}$ such that, if $i = \dim \widetilde \sigma$,
    $$ \lip (\widetilde h_{ \widetilde \sigma}) \le b_i(\lip(\widetilde \sigma)); $$
    \item[(3)] if $\{\widetilde \sigma_j\}_{j \in \N}$ is a locally finite family of simplices in $X$, then the family of sets $\{\widetilde h_{\widetilde \sigma_j}(\Delta^* \times I)\}_{j \in \N}$ is locally finite in $X$.
\end{enumerate}
Then we have well defined maps:
$$ f_*^{(m)} \: C_*^{\lf, \lip, (\ell^1)}(X) \to C_*^{\lf, \lip, (\ell^1)}(X) \qquad f^{(m)}(\widetilde \sigma) := \widetilde h_{\widetilde \sigma} \circ i_m \quad m = 0, \, 1 $$ 
that are $C_*^{\lf, \lip, \ell^1}$-chain-homotopic.

Moreover, the obvious smooth version of the statement is also true.
\end{lem}

\begin{proof} From (1) and \cite[Lemma 2.13]{LS} we have, for all $k \le i+1$, affine inclusions $G_{k, i} \: \Delta^{i+1} \to \Delta^i \times I$ such that the map
\begin{equation} \label{3oiqjfwkl} \tilde H \: C_i^\lip(\widetilde N) \to C_{i+1}^\lip(\widetilde N) \qquad \tilde H(\tilde \sigma^{(i)}) := \sum_{k=0}^i \tilde h_{\tilde \sigma} \circ G_{k, i} \end{equation}
provides a $C_*^\lip$--chain homotopy between $f^{(0)}$ and $f^{(1)}$. 
Hence (1) is proven. The hypotheses (2) and (3) and the form of \eqref{3oiqjfwkl} easily allow us to extend this homotopy to a $C_*^{\lf, \lip, \ell^1}$--homotopy between $f^{(0)}$ and $f^{(1)}$.
\end{proof}

\begin{cor} \label{p37yqhwla} The $\Lambda$--equivariant map $\st_* \: C_*^{\sm}(\widetilde N) \to C_*^{\sm}(\widetilde N)$ induces a well defined chain map:
$$ \st_* \: C_*^{\sm}(N) \to C_*^{\sm}(N) $$
chain-homotopic to the identity.
\end{cor}

If $\{\widetilde \sigma_j\}_{j \in \N}$ is a uniformly Lipschitz family of simplices in $\widetilde N$, then $\{\st(\sigma_j)\}_{j \in \N}$ is locally finite and uniformly Lipschitz. Moreover, if $\{\widetilde \sigma_j\}_{j \in \N}$ is locally finite, then the family of homotopies $\{\widetilde h_{\widetilde \sigma_j}\}_{j \in \N}$ is locally finite too. Observe that a $\Lambda$--equivariant locally finite family of uniformly Lipschitz homotopies in $\tilde N$ projects onto a locally finite family of uniformly Lipschitz homotopies. Therefore we can apply the second part of Lemma \ref{da chain top a chain homo pag 11 LS}, and obtain:


\begin{cor} \label{96t347yoehiluwdjs} The $\Lambda$--invariant map: $\st_* \: C_*^{\lf, \ell^1, \lip}(\widetilde N) \to C_*^{\lf, \ell^1, \lip}(\widetilde N)$ induced by $\st$ is $C_*^{\lf, \ell^1, \lip}$--chain-homotopic to the identity. By the $\Lambda$--equivariance of the homotopies $\tilde h_{\tilde \sigma}$, the induced map: $\st_* \: C_*^{\lf, \ell^1, \lip}(N) \to C_*^{\lf, \ell^1, \lip}(N)$ is also chain-homotopic to the identity.
\end{cor}

In particular, since $\st_* (C_*^{\lf, \ell^1, \lip}(N)) \subseteq C_*^{\lf, \ell^1, \lip, \sm}(N)$ and $\st_*$ is norm non-increasing, we immediately get: 
\begin{cor} \label{3o7hfuweidj} The Lipschitz simplicial volume can be computed by smooth cycles: 
$$ \|N\|_\lip = \inf \{\|c_N\|_1 \: c_N \in C_n^{\lf, \lip, \ell^1, \sm}(N) \textrm{ is a (smooth) fundamental cycle} \} $$
where we put $\inf \emptyset = +\infty$.
\end{cor}


\section{\label{sezione 4} Proof of the Proportionality Principle for non-compact manifolds}

We are now ready to prove our main result. Our argument retraces the proof in \cite{LS} almost verbatim. In order to get rid of the assumptions about curvature of \cite{LS}, we need to replace the straightening operator defined in \cite{LS} with our pseudostraightening procedure, described in Theorem \ref{pseudostraight:thm}. 


We recall some relevant definitions from \cite[p. 50]{L} and \cite[p. 24]{LS}. Given a signed measure set $(Z, \mu)$, we say that a subset $Z' \subseteq Z$ is a \textbf{determination set}
 for $\mu$ if $\mu(V) = 0$ for every measurable set $V \subseteq Z \setminus Z'$. Let $X$ be a Riemannian manifold, and let $\mu$ be a signed measure of finite total variation on the set $Z := S_i^\sm(X)$, which is Borel with respect to the $C^1$--topology of $S_i^\sm(X)$. 
We say that $\mu$ has \textbf{Lipschitz determination} if $S_i^{L, \sm}(X)$ is a determination set of $\mu$ for some $L \ge 0$.
We denote by $ \mathscr C_*^\lip(X)$ the complex of \textbf{Lipschitz measure chains}, i.e. the complex of signed Borel measures on $S_*(X)$ of Lipschitz determination and finite variation, 
with the natural boundary operator. We denote by $ \mathscr H_*^\lip(X)$ the associated \textbf{Lipschitz measure homology}.

\subsection*{Smearing}
Let $M$ and $N$ be $n$--dimensional manifolds as in Theorem \ref{main theorem}. Let $\pi_1(N) = \Lambda$, $G := \isom^+ \,\, \widetilde N$ and $\pi_1(M) = \Gamma$. 
Under the assumption that $M$ and $N$ have isometric Riemannian universal covers, we have a \emph{smearing} chain map:
\begin{equation} \label{efhiuwljn} \smear_* \: C_*^{\lf, \ell^1, \lip, \sm}(M) \to \mathscr C_*^\lip(N) \end{equation}
which is defined as follows. Fix on $G$ the compact-open topology. It is well-known that $G$ is a locally compact topological group (in fact, it is a Lie group), so it admits a left-invariant Haar measure $|\cdot|_G$. It is proven in \cite{LS} that $\Lambda$ is a lattice in $G$, hence $G$ is unimodular, so $|\cdot|_G$ is also right-invariant. We denote by $|\cdot|_{\Lambda \backslash G}$ the right-invariant measure induced by $|\cdot|_G$ on $\Lambda \backslash G$, normalized in such a way that $|\Lambda \backslash G|_{\Lambda \backslash G}=1$.

We normalize $|\cdot|_{\Lambda \backslash G}$ to 1. By the Haar uniqueness Theorem, $|\cdot|_{\Lambda \backslash G}$ is uniquely defined. For a simplex $\sigma \in S_*^{\sm}(M)$ which lifts to some $ \widetilde \sigma \in S_*^\sm(\tilde N)$ the element $\smear(\sigma)$ is defined to be the pushforward of $|\cdot|_{\Lambda \backslash G}$ through the map:
$$ \Lambda \backslash G \to S_*^\sm(N) \qquad  [g] \mapsto p_N \circ g \circ \widetilde \sigma. $$
We extend $\smear$ over uniformly Lipschitz $\ell^1$--sums by linearity. The map $\smear$ is chain and well defined by the right invariance of the Haar measure. 

Now, consider the diagram below:

\begin{equation} \label{fehiruj}
\begin{diagram} 
&&   \mathscr  C_*^{\Lip}(N)\\
& \ruTo^{\smear_*} & \phantom{\,\, j_*} \uTo \,\, j_* \\ 
C_*^{\lf, \ell^1, \Lip, \sm}(M) & \xrightarrow{\phantom{spa} \varphi_*\phantom{zio} } & C_*^{\lf, \ell^1, \Lip, \sm}(N) \\ 
\end{diagram}
\end{equation}
Here, $j$ is the natural inclusion, defined by mapping a simplex $\sigma$ in $N$ to the Dirac measure $\delta_\sigma$ centered in $\sigma$. We want to define a \emph{discrete smearing} $\varphi_* \: C_*^{\lf, \ell^1, \lip, \sm}(M) \to C_*^{\lf, \ell^1, \lip, \sm}(N)$, i.e. a norm non-increasing chain map which makes diagram \eqref{fehiruj} commute up to homotopy.

\begin{ntz} \label{} In this section, we will attain ourselves to the following notation:
$$ \widetilde h_{ \widetilde \sigma}(g) := \widetilde h_{g \widetilde \sigma} \qquad h_{ \widetilde \sigma}(g) := p_N \circ \widetilde h_{g \widetilde \sigma} \qquad f_{\widetilde \sigma}(g) := p_N \circ \st( g \widetilde \sigma). $$
\end{ntz}

The map $G \to S_*^\sm(\widetilde N)$, $g \mapsto g \widetilde \sigma$ (where $G$ is equipped with the compact-open topology and $S_*^\sm(\widetilde N)$ with the $C^1$--topology) is continuous, hence Borel (see \cite[Section 4.3.2]{LS}). Therefore, by point (5) of Theorem \ref{pseudostraight:thm}, the maps $g \mapsto \widetilde h_{ g \widetilde \sigma}$ and $g \mapsto h_{ g \widetilde \sigma}$ are Borel too. In particular, $f_{\tilde \sigma}$ is Borel. Taking the quotient by $\Lambda$ we get a Borel map which we still denote by $f_{\widetilde \sigma}$:
$$ f_{\widetilde \sigma} \: \Lambda \backslash G \to S_i^\sm(N), \qquad  [g] \mapsto p_N \circ \st_i(g \widetilde \sigma). $$

We define the \textbf{discrete smearing $\varphi$} as:
\begin{equation} \label{definizione di varphi} \varphi_i \: C_i^{\lf, \ell^1, \lip, \sm}(M) \to C_i^{\lf, \ell^1, \lip, \sm}(N) \qquad \sum_{j \in \N} \lambda_j \sigma_j \mapsto \sum_{j \in \N} \lambda_j \sum_{\varrho \in \widehat{S_i}(N)} |f_{\widetilde \sigma_j}^{-1}(\varrho)|_{\Lambda \backslash G} \varrho, \end{equation}
where $ \widetilde \sigma_j \in S_i^\sm(\widetilde N)$ lifts $\sigma_j$.

We now show that $\varphi_*$ is well defined -- i.e. that \eqref{definizione di varphi} is independent of the choice of the lifts $\widetilde \sigma_j$ of $\sigma_j$ in $M$ -- and that the chain on the right side of \eqref{definizione di varphi} is Lipschitz and locally finite. 
The first assertion is a straightforward consequence of the right invariance of the Haar measure on $\Lambda \backslash G$. 

Now, let $L \in \R$ be such that $\lip (\sigma_j) = \lip ( \widetilde \sigma_j) \le L$ for all $j$. By point (3) of Theorem \ref{pseudostraight:thm}:
$$ \Lip(f_{\widetilde \sigma_j}(g)) = \lip (\st (g \widetilde \sigma_j)) \le b_i(\lip (g \widetilde \sigma_j)) = b_i(\lip (\widetilde \sigma_j)) \le b_i(L), $$ 
and this readily implies that the chain $\varphi_i(\sum_j \lambda_j \sigma_j)$ is Lipschitz.

The local finiteness is implied by point (3) of Theorem \ref{pseudostraight:thm}. The verification that $\varphi_*$ is chain is identical to the one given in \cite[Section 4.3.2]{LS}.

We now prove that \eqref{fehiruj} commutes up to homotopy. We need a technical Lemma that allow us to transfer $C_*^\lip$--homotopies to homotopies in measure homology. The proof of the following lemma is identical to the one given in \cite[Section 2.6]{Zastrow}, and we simply note that it works with much weaker hypotheses. 


\begin{lem} [\cite{Zastrow}] \label{Generale Zastrow} Let $(X, \Sigma_X)$ and $(Y, \Sigma_Y)$ be measurable spaces, and let 
$$ \{f_i\}_{i  = 1}^{k} \: X \to Y $$ 
be a finite number of measurable maps between them. Let $ \mathscr C(X)$ be the vector space of signed $\Sigma_X$--measures on $X$ with bounded total variation. The same with $Y$. Let $\lambda_i \in \R$ be such that:
$$ \sum_{i=1}^k \lambda_i f_i(x) = 0 \qquad \forall x \in X $$
(this is a formal linear combination of elements of $Y$). Then, if $ \mathfrak f_i \: \mathscr C(X) \to \mathscr C(Y)$ denotes the pushforward of $f_i$, we have:
$$ \sum_i \lambda_i\,\, \mathfrak f_i(\mu) = 0 \qquad \forall \mu \in \mathscr C(X). $$
\end{lem}
$\\\\$

Since $h_{\widetilde \sigma} \: G \to C^1(\Delta^i \times I, N)$ is Borel, the following map, which we denote with the same name, is Borel too:
$$ h_{\widetilde \sigma} \: \Lambda \backslash G \to C^1(\Delta^i \times I, N). $$

For every simplex $\widetilde \sigma \in C_*^\sm(\widetilde N)$ and $g \in G$, the element $\widetilde h_{\widetilde \sigma}(g)$ is a smooth homotopy between $\st (g \widetilde \sigma)$ and $g \widetilde \sigma$. By point (1) of Theorem \ref{pseudostraight:thm}
and Lemma \ref{da chain top a chain homo pag 11 LS}, the map:
\begin{equation} \label{28037ajos} \widetilde H \: C_i^\sm(\widetilde N) \to C_{i+1}^\sm(\widetilde N) \qquad \widetilde H_{\widetilde \sigma^{(i)}}(g) = \sum_{k=0}^i \widetilde h_{\widetilde \sigma}(g) \circ G_{k, i} \end{equation}
defines a $\Lambda$--equivariant chain homotopy between 
$\widetilde \sigma \mapsto \st_i (g \widetilde \sigma)$ and $\widetilde \sigma \mapsto g \widetilde \sigma$ in $C_i^\sm(\widetilde N)$
. 
Let $H := p_N \circ \widetilde H$. Composing \eqref{28037ajos} with $p_N$, we get a chain homotopy:
$$ H \: C_i^\sm(\widetilde N) \to C_{i+1}^\sm(N) \qquad \widetilde \sigma \mapsto H_{\widetilde \sigma}(g) = \sum_{k=0}^i h_{\widetilde \sigma}(g) \circ G_{k, i} $$
between $\widetilde \sigma \mapsto f_{\widetilde \sigma}(g) = p_N \circ \st(g\widetilde \sigma)$ and $ \widetilde \sigma \mapsto p_N \circ g \widetilde \sigma$ in $C_*^\sm(\widetilde N)$. More explicitely, and using the $\Lambda$--invariance of the maps involved, we have:
\begin{equation} \label{3ofqhiuewl} f_{ \widetilde \sigma}(\Lambda g) - p_N \circ \Lambda g \widetilde \sigma = 
\partial_{i+1} \left( \sum_k h_{ \widetilde \sigma}(\Lambda g) \circ G_{k, i} \right) + \sum_{k=0}^{i-1} \sum_{j=0}^i  h_{\partial_{j} \widetilde \sigma}(\Lambda g) \circ G_{k, i-1} \qquad \forall \widetilde \sigma \in S_i^\sm(\widetilde N). \end{equation}
Let $ \widetilde \sigma \in S_i^\sm(\tilde N)$ be the lift of some $\sigma \in S_i^\sm(M)$. By $\nu_{\sigma, k,i}$ we denote the pushforward of $|\cdot|_{\Lambda \backslash G}$ through the map:
$$ \Lambda \backslash G \to S_{i+1}^\sm(N) \qquad \Lambda g \mapsto h_{\widetilde \sigma}(g) \circ G_{k, i}. $$
By the right invariance of the Haar measure on $\Lambda \backslash G$, it is easy to see that $\nu_{\sigma, k, i}$ is well defined, i.e. it is independent of $\tilde \sigma$.

By Lemma \ref{Generale Zastrow}, with $X = \Lambda \backslash G$ and $Y = S_i^\sm(N)$, we are authorized to pushforward $|\cdot|_{\Lambda \backslash G}$ through every map in \eqref{3ofqhiuewl} and sum:
\begin{equation} \label{oy3efhilqnwj} j_i(\varphi_i(\sigma)) - \smear_i(\sigma) = \partial \sum_k \nu_{\sigma, k, i} + \sum_j \sum_k \nu_{ \partial_j \sigma, k, i-1}, \end{equation}

It follows from \eqref{oy3efhilqnwj} that the map:

$$ \mathscr H_* \: C_*^{\lf, \ell^1, \lip, \sm}(M) \to \mathscr C_{*+1}^{\lip}(N) \qquad \mathscr H_i \left(\sum_{j \in \N} \lambda_j \sigma_{j} \right) := \sum_{j \in \N} \lambda_j \sum_{k=0}^i \nu_{\sigma_j, k, i} $$
provides a chain homotopy between $j_* \circ \varphi_*$ and $\smear_*$. Indeed, the measure on the right has finite total variation, since it is an $\ell^1$--sum of measures whose total variation is 1. Moreover, it obviously has Lipschitz determination set
.




\subsection*{The conclusion of the proof}
In order to conclude, we now exploit some results from \cite{LS}. 
By pairing with the volume form, the following two functions are introduced in \cite{LS}:
$$ \left\langle \dvol ; \,\, \right\rangle \: \, C_n^{\lf, \ell^1, \lip, \sm}(N) \to \R \qquad \left\langle \dvol; \sigma \right\rangle := \int_{\Delta_n} \sigma^* \dvol, $$
$$ \left\langle \dvol ; \,\, \right\rangle \: \mathscr C_n^{\lip}(N) \to \R \qquad \left\langle \dvol; \mu \right\rangle := \int_{C^1(\Delta, N)} \left\langle \dvol, \sigma \right\rangle d\mu(\sigma), $$ 
where $\sigma \in C^1(\Delta_k, N)$, and $\mu \in \mathscr C_n^{\sm}(N)$. Note that the function ~$C^1(\Delta, N) \ni \sigma \mapsto \left\langle \dvol, \sigma \right\rangle$ is continuous, hence in particular $C^1(\Delta, N)$--Borel. However, it is probably not even Borel with respect to the $d_\infty$--topology, and this is the reason why we needed to concentrate our attention on spaces of smooth simplices endowed with the $C^1$--topology (see also \cite[p. 108]{L}). By an application of Stokes Theorem, we have well defined functions $H_n^{\lf, \lip, \ell^1, \sm}(N) \to \R$ and $ \mathscr H_n^\lip(N) \to \R$.

Those functions satisfy the following properties: if $\|M\|_\lip < \infty$, then \begin{enumerate}
    \item $ \left\langle \smear_n (c_M), \dvol_N \right\rangle = \vol(M) $, where $\dvol_N$ is the volume form on $N$ and $c_M$ is a fundamental cycle of $M$ in $C_n^{\lf, \ell^1, \lip, \sm}(M)$.
    \item $\left\langle c, \dvol_N \right\rangle = \left\langle j_n(c), \dvol_N \right\rangle$, for every cycle $c \in C_*^{\lf, \ell^1, \lip, \sm}(N)$. Here, $j \: C_n^{\lf, \ell^1, \lip, \sm}(M) \to \mathscr C^{\lip, \sm}(M)$ is the natural inclusion.
    \item $ \left\langle c_N, \dvol_N \right\rangle = \vol(N)$, if (and only if) $c_N \in C_*^{\lf, \ell^1, \lip, \sm}(N)$ is a fundamental cycle of $N$.
\end{enumerate}

Point (3) and (1) are proven in \cite[Proposition 4.4, Proposition 4.10]{LS}, while (2) is easily seen to be true. 

We now complete the proof of Theorem \ref{main theorem}. By symmetry, it is sufficient to prove that
$$ \|M\|_\lip \ge \frac{\vol M}{ \vol N} \|N\|_\lip; $$
hence we can assume that $\|M\|_\lip < \infty$. By Corollary \ref{3o7hfuweidj}, for every $ \varepsilon>0$ there exists a fundamental cycle $c_M \in C_n^{\lf, \lip, \ell^1, \sm}(M)$ of $M$ such that $\|c_M\| \le \|M\|_\lip + \varepsilon$. From (3) and
$$ \left\langle \dvol_N,\, \varphi_n(c_M) \right\rangle = \left\langle \dvol_N,\, j_n\circ \varphi_n(c_M) \right\rangle = \left\langle \dvol_N,\, \smear_n(c_M) \right\rangle = \vol N $$
it follows that $\varphi_n(c_M)$ is the $ \frac{\vol M}{ \vol N}$ multiple of a fundamental cycle of $N$. But the map $\varphi_*$ is norm non-increasing, hence 
$$  \frac{\vol M}{ \vol N} \|N\|_\lip \le \|\varphi_n(c_M)\| \le \|c_M\| \le \|M\|_\lip + \varepsilon,  $$
whence the conclusion.

\section{\label{sezione 5} Supermultiplicativity of the Lipschitz simplicial volume of products}



In \cite{LS}, by using the straightening procedure, Theorem \ref{productEstimates} is proven for non-positively curved manifolds. From the proof of L\"oh and Sauer, it is apparent that the only place where the hypotheses about curvature are exploited is the proof of \cite[Proposition 3.20]{LS}. Hence it is sufficient to prove Proposition \ref{3o87rdhiuwljk}, which is the analogous of \cite[Proposition 3.20]{LS}, without the curvature condition.

\begin{dfn}[\cite{LS}] \label{ } Let $M$ and $N$ be Riemannian manifolds, and let $\pi_M \: M \times N \to M$ and $\pi_N \: M \times N \to N$ be the projections. A \textbf{sparse} chain on $M \times N$ is a chain $c = \sum_{k \in \N} \lambda_k \,\, \sigma_k \in C_*^{\lf}(M \times N)$ such that $\sum_{k \in \N} \lambda_k \,\, \pi_M \circ \sigma_k \in C_*^{\lf}(M)$ and $\sum_{k \in \N} \lambda_k \,\, \pi_N \circ \sigma_k \in C_*^{\lf}(N)$.
\end{dfn}

\begin{prop} \label{3o87rdhiuwljk} Let $M$ and $N$ be complete Riemannian manifolds. Then, for every $c \in C_*^{\lf, \lip}(M \times N)$, there is a \emph{sparse} cycle $c' \in C_*^{\lf, \lip}(M \times N)$ homologous to $c$ with $\|c'\| \le \|c\|$.
\end{prop}

\begin{proof} Let $p_N \: \widetilde N \to N$ be the Riemannian universal covering of $N$, and let 
$ {\st_N}_* \:  \\ S_*^\lip(N) \to S_*^\lip(N) $ be the map given by
$$ \st_N(\sigma) := p_N \circ \st(\tilde \sigma) $$
where $\tilde \sigma$ is a lift of $\sigma$ in $\widetilde N$. Analogously we define $\st_M$. We have a well defined map:
$$ (\st_M, \st_N)_* \: C_*^{\lf, \lip}(M \times N) \to C_*^{\lf, \lip}(M \times N) \qquad \sigma \mapsto \left(\st_M(\pi_M\circ\sigma), \st_N(\pi_N\circ\sigma)\right). $$
For every simplex $\sigma \in S_*^\lip(M)$, let $h^M(\sigma)$ be a homotopy between $\st(\sigma)$ and $\sigma$ as in Theorem \ref{pseudostraight:thm}. Analogously with $N$. It is easily seen that the system of homotopies $(h^M(\sigma), h^N(\tau)$), for $(\sigma, \tau) \in S_*^\lip(M \times N) = S_*^\lip(M) \times S_*^\lip(N)$, satisfies conditions (1), (2), (3) of Lemma \ref{da chain top a chain homo pag 11 LS}. Therefore 
$$ (\st_M, \st_N)_* \: C_*^{\lf, \lip}(M \times N) \to C_*^{\lf, \lip}(M \times N) $$ 
is chain homotopic to the identity, norm non-increasing, and maps cycles to sparse cycles. 
Hence the conclusion follows.
\end{proof}

COMMENTATA APPLICAZIONE ISOMETRICA

\bibliographystyle{amsalpha}
\bibliography{biblio}

\end{document}